\documentclass{article}
\usepackage{graphicx, mathtools, amsmath, amsthm, amssymb, bbm, xcolor, hyperref, todonotes} 

\DeclarePairedDelimiter{\bigabs}{\bigg|}{\bigg|}
\DeclarePairedDelimiter{\abs}{\lvert}{\rvert}
\DeclarePairedDelimiter{\paren}{\lparen}{\rparen}
\DeclarePairedDelimiter{\inprod}{\langle}{\rangle}
\DeclarePairedDelimiter{\biginprod}{\bigg\langle}{\bigg\rangle}
\DeclarePairedDelimiter{\norm}{\lVert}{\rVert}

\newcommand{\hypgeo}[2]{{\vphantom{F}}_{#1}\kern-\scriptspace F_{#2}}
\numberwithin{equation}{section}
\numberwithin{equation}{section}
\newtheorem{theorem}{Theorem}
\newtheorem{lemma}{Lemma}

\newtheorem{prop}{Proposition}

\newtheorem*{remark}{Remark}

\newcommand{\N}{\mathbb N}
\newcommand{\Z}{\mathbb Z}
\newcommand{\R}{\mathbb R}

\newcommand{\Bk}{\color{black}}

\title{\bf The magnetic Laplacian on the Disc\\ for strong magnetic fields}
\author{Ayman Kachmar$^*$ and Germ\'an Miranda$^\dag$ \\
\footnotesize{\it $^*$The Chinese University of Hong Kong, Shenzhen, Guangdong,
518172, P.R. China}\\
\footnotesize{akachmar@cuhk.edu.cn}\\
\footnotesize{\it $^\dag$Department of Mathematics, Lund University, Box 118, 22100, Sweden}\\
\footnotesize{german.miranda@math.lth.se}}

\date{\today}
\begin{document}
\maketitle
\begin{abstract}
    The magnetic Laplacian on a planar domain under a strong constant magnetic field has eigenvalues close to the Landau levels. We study the case when the domain is a disc and the spectrum consists of branches of eigenvalues of one dimensional operators. Under Neumann boundary condition and strong magnetic field, we   derive asymptotics of the eigenvalues with  accurate estimates of  exponentially small remainders. Our approach is purely  variational and applies to the Dirichlet boundary condition as well, which allows us to recover   recent results by Baur and Weidl.
\end{abstract}

\section{Introduction}\label{sec:int}

\subsection{Motivation}
In their recent paper \cite{baurweidl}, Baur and Weidl studied the  Laplacian on the disc with homogeneous magnetic field and Dirichlet boundary condition. In the strong magnetic field limit, they proved that the branches of the eigenvalues behave linearly with respect to the magnetic field strength with  exponentially small remainder. Their proof relies on the analysis of the corresponding ordinary differential equations with respect to the radial variable, and a key ingredient in their analysis involves the properties of special functions.

In this paper, we address the same question but for the Neumann boundary condition, and we derive analogous   eigenvalue asymptotics. To leading order, the exponentially small remainders in our setting are up to sign change the same as those in \cite{baurweidl}. However, the approach via special functions as in \cite{baurweidl} seems challenging  for the Neumann boundary condition. We therefore propose a different variational approach based on the Temple inequality, which interestingly works for the Dirichlet and Robin boundary conditions  too.

\subsection{The magnetic Laplacian on the disc}
Let \(b>0\) be the strength of the magnetic field and consider the Dirichlet or Neumann realization of the magnetic Laplacian in the domain \(\Omega_R := D(0,R) = \{x\in \mathbb{R}^2 : \abs{x} <R\}\),
\begin{equation}
    H_b(R) = (-i\nabla +bA)^2, 
\end{equation}
where \(A(x_1, x_2) = \frac{1}{2}\paren{-x_2, x_1}\) is the vector magnetic potential generating the unit magnetic field $\mathrm{curl}\,A=1$. Due to the rotational symmetry of the problem it is helpful to consider polar coordinates
\[\begin{cases}
x_1 = r\cos{\theta} \\
x_2 = r\sin{\theta}
\end{cases}\]
where \(r\in [0, +\infty)\) and \(\theta\in [0, 2\pi )\). In these coordinates, the vector potential reads 
\begin{equation*}
    A(x_1, x_2) = \frac{r}2 (-\sin{\theta} , \cos{\theta}) = \frac{r}{2} \hat{e}_{\theta}, 
\end{equation*}
and the operator \(H_b(R) \) expresses as 
\begin{equation*}
   H_b(R) = -\partial^2_r - \frac{1}{r}\partial_r +\bigg (\frac{i}{r}\partial_{\theta} - \frac{br}{2}\bigg )^2.
\end{equation*}
In polar coordinates, the Dirichlet and Neumann boundary conditions are, respectively, $u|_{r=1}=0$ and $\partial_ru|_{r=1}=0$.\medskip

The operator \(H_b(R)\) is unitarily equivalent to \(R^{-2}H_{bR^2}(1)\) via  the \(L^2\)-unitary transform
   \begin{equation*}
       U_R u(r, \theta) := u(r/R, \theta),
   \end{equation*}
so there is no loss of generality to restrict the study to the unit disc. Hereafter, we will deal with \(H_b(1)\), and we will write \(H_b\) and \(\Omega\) instead of \(H_b(1)\) and \(\Omega_1\) to lighten the notation. 

We denote by $H_b^D$ and $H^N_b$ the Dirichlet and Neumann realizations of $H_b$, with domains 
\[\begin{split}
    D(H_b^D)=&\{u\in L^2(\Omega) \colon H_bu\in L^2(\Omega),\quad u|_{\partial\Omega}=0\},\\
    D(H_b^N)=&\{u\in L^2(\Omega) \colon H_bu\in L^2(\Omega),\quad \partial_r u|_{\partial\Omega}=0\}.\\
\end{split}\]
\subsection{Fourier decomposition}
We can do a Fourier decomposition of \(L^2(\Omega)\) as follows
\begin{equation*}
    L^2(\Omega)\cong L^2((0, 1), rdr) \otimes L^2 (\mathbb{S}^1, d\theta) \cong \bigoplus_{m\in \mathbb{Z}}\biggl( L^2((0, 1), rdr) \otimes  \biggl[\frac{e^{-im\theta}}{\sqrt{2\pi}}\biggr]\biggr),
\end{equation*}
by decomposing \(u\in L^2(\Omega)\) as
\begin{equation*}
    u(r, \theta) = \sum_{m\in \mathbb{Z}} u_m(r) \frac{e^{-im\theta}}{\sqrt{2\pi}},
\end{equation*}
where 
\begin{equation*}
u_m(r) =   \frac{1}{\sqrt{2\pi}}\int_0^{2\pi} u(r, \theta) e^{im\theta}\, d\theta, 
\end{equation*}
and \(\sum_{m\in \mathbb{Z}} \norm{u_m}^2 < +\infty\). This decomposition allows us to decompose our original operators as

\begin{equation*}
    H_b^D \ = \bigoplus_{m\in \mathbb{Z}}\biggl( H_{m,b}^D \otimes I_m\biggr) \text{ and } H_b^N \ = \bigoplus_{m\in \mathbb{Z}}\biggl( H_{m,b}^N \otimes I_m\biggr)
\end{equation*}
where  \(I_m\) is the identity operator on \(e^{-im\theta}/\sqrt{2\pi}\). Moreover, \(H_{m,b}^D\) and \(H_{m,b}^N\) are the Dirichlet, respectively Neumann, realization acting on \(L^2((0, 1), r\, dr) \)  of the differential operator
\begin{equation}\label{Hm,b}
   H_{m,b} = -\frac{d^2}{dr^2} - \frac{1}{r}\frac{d}{dr} +\bigg (\frac{m}{r} - \frac{br}{2}\bigg )^2 .
\end{equation}
More precisely,  for $m\not=0$, the domains of $H^D_{m,b}$ and $H^N_{m,b}$ are 
\begin{align*}
  \mathcal{D}(H_{m,b}^D)=&\{u\colon u,u/r,u',H_{m,b}u\in L^2((0, 1), r\, dr),~  u(1)=0\},\\
    \mathcal{D}(H_{m,b}^N) = &\{u\colon u,u/r,u',H_{m,b}u\in L^2((0, 1), r\, dr),~ u'(1)=0\}, 
\end{align*}
whereas\footnote{Using the Liouville transform $L^2(\R_+,rdr)\ni u(r)\mapsto f(r)=\sqrt{r}u(r)\in L^2(\R_+,dr)$ and \cite[Theorem 2.2]{KristenLoya}, we can show that if \(u\in\mathcal{D}(H_{0,b}^\#)\), $\#\in\{D,N\}$, then  $u'(0)=0$.} (see \cite[Remark 3.1]{Isoperimetric})
\begin{align*}
    \mathcal{D}(H_{0,b}^D)=&\{u\colon u,H_{0,b}u\in L^2((0, 1), r\, dr),~  u(1)=0 \text{ and } \lim_{r\rightarrow 0^+} \frac{u(r)}{\ln{r}} =0 \},\\
    \mathcal{D}(H_{0,b}^N)=&\{u\colon u,H_{0,b}u\in L^2((0, 1), r\, dr),~ u'(1)=0 \text{ and } \lim_{r\rightarrow 0^+} \frac{u(r)}{\ln{r}} =0\}.
\end{align*}

\subsection{Main result}
Since we are considering a bounded smooth domain in \(\mathbb{R}^2\), we know that both the Dirichlet realization \(H^D_{m,b}\) and the Neumann realization \(H_{m,b}^N\) of $H_{m,b}$ have discrete spectra. For $n\in\N$, let \(\lambda_n\paren{H^D_{m,b}}\) and   \(\lambda_n\paren{H^N_{m,b}}\) be the $n$-th eigenvalue of 
$H_{m,b}^D$ and $H^N_{m,b}$ respectively, counting multiplicity.
These are the branches of eigenvalues that constitute the spectra of the Dirichlet and Neumann realizations of the magnetic Laplacian,
\[\mathrm{sp}(H^D_b)=\bigcup_{m\in\Z,n\in\N}\{\lambda_n\paren{H^D_{m,b}}\}\quad\mbox{and}\quad \mathrm{sp}(H^N_b)=\bigcup_{m\in\Z,n\in\N}\{\lambda_n\paren{H^N_{m,b}}\}.\]
Let us recall the result of \cite{baurweidl} concerning the eigenvalues of $H^D_b$, which asserts that for  $n$ and $m$ fixed, we have as $b\to+\infty$,
\begin{multline}\label{eq:BW}
    \lambda_n\paren{H^D_{m,b}} =
    \paren{2n-1 +\abs{m} -m} b  \\+e^{-\frac{b}{2}} \bigg( \frac{ b^{2n+m}  }{  \paren{n-1}!\paren{m+n-1}!2^{2(n-1)+m}} + \mathcal{O}\paren{b^{2n+m}}\biggr).
\end{multline}
This behaviour is displayed in Figure \ref{Figure 2}. Notice that, when $n=1$, the asymptotics in \eqref{eq:BW} was obtained earlier by Helffer and Sundqvist \cite[Theorem~5.1]{HelfferSund} for any $m\in\Z$, and by \cite{erdos} for $m=0$.\medskip

 For the Neumann realization, numerical computations (see \cite{SaintJames1965EtudeDC} or Figure \ref{Figure 1} below) suggest that the branches of the eigenvalues behave linearly with respect to the magnetic field strength but with a negative exponentially small remainder. Our main result concerns the eigenvalues of the Neumann realization $H^N_b$, which we state below.
\begin{theorem}\label{Main theorem}
    Given $n\in\N$ and $m\in\Z$, if \(b\rightarrow +\infty\), it holds that
    \begin{multline}\label{eq:main}
        \lambda_n\paren{H^N_{m,b}} =  \paren{2n-1 +\abs{m} -m} b\\  - e^{-\frac{b}{2}} \bigg( \frac{ b^{2n+m}  }{  \paren{n-1}!\paren{m+n-1}!2^{2(n-1)+m}} + \mathcal{O}\paren{b^{2n+m}}\biggr). 
    \end{multline}
\end{theorem}

\begin{figure}
\includegraphics[width=12cm]{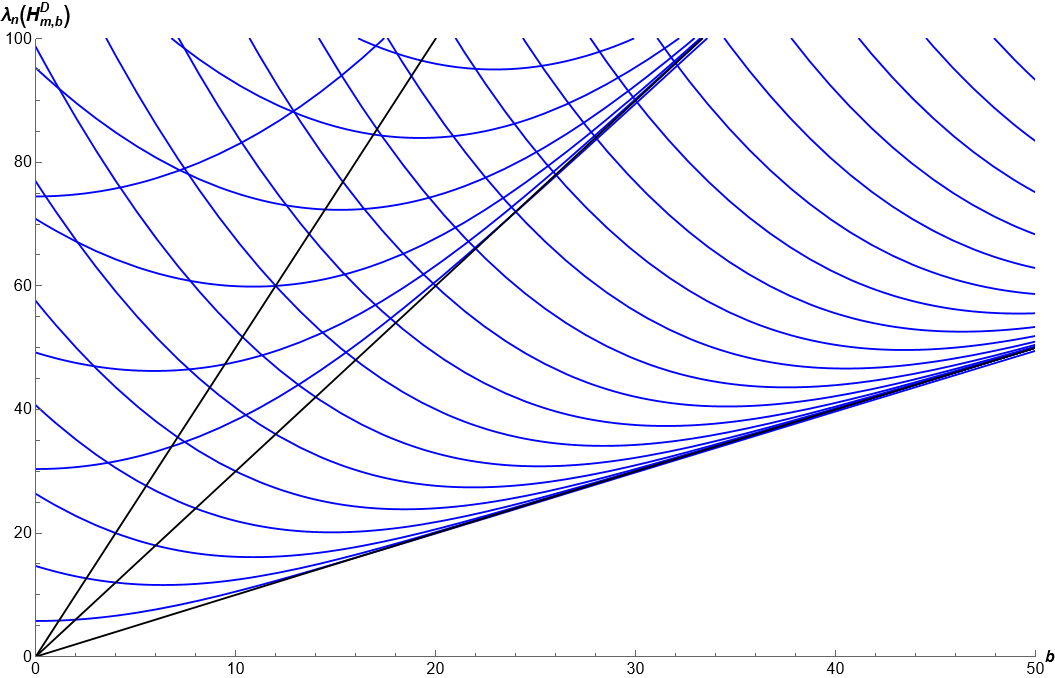}
\centering
\caption{Numerical computations of \( \lambda_n\paren{H^D_{m,b}} \) (blue) for \(0\leq m \leq  25\) and \(1\leq n \leq 3\) together with the curves \(b,3b\) and \(5b\) in black.}
\label{Figure 2}
\end{figure}
\begin{figure}
\includegraphics[width=12cm]{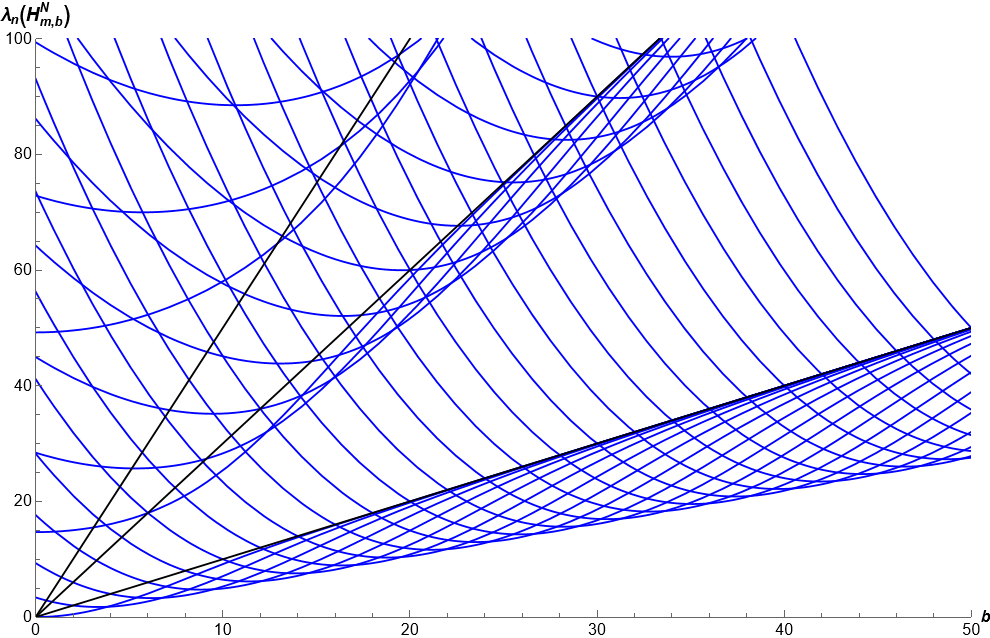}
\centering
\caption{Numerical computations of \( \lambda_n\paren{H^N_{m,b}} \) (blue) for \(0\leq m \leq  25\) and \(1\leq n \leq 3\) together with the curves \(b,3b\) and \(5b\) in black.}
\label{Figure 1}
\end{figure}
\begin{remark}\rm 
Our proof of Theorem~\ref{Main theorem} relies on the construction of a trial state and allows us to recover \eqref{eq:BW}. The proof of \eqref{eq:BW} in  \cite{baurweidl} relies on  special functions, and if we try to adapt it  for the Neumann case, we end up with finding \(\lambda\) such that
    \begin{equation}\label{Neumann case with special function}
    2bM'\biggl(\frac{1}{2}\paren{ 1 - \frac{\lambda}{b}},m + 1, \frac{b}{2} \biggr) -(b-2m) M\biggl(\frac{1}{2}\paren{ 1 - \frac{\lambda}{b}},m+ 1, \frac{b}{2} \biggr) =0,
    \end{equation}
    where \(M(a,c,z)\) is the Kummer's function\footnote{See (13.2.1) at \href{http://dlmf.nist.gov/13.2.E1}{\tt http://dlmf.nist.gov/13.2.E1}.} with \(a=\frac{1}{2}\paren{ 1 - \frac{\lambda}{b}}\) and \(c= m+ 1 \). Solving this equation is way harder than for the Dirichlet case, where one needs to find roots of a given Kummer's function \(M(a,c,z)\),  which is done in \cite{baurweidl}. 
\end{remark}

\subsection{The trial state and scheme of the proof}

We will divide the proof of Theorem \ref{Main theorem} into finding matching upper and  lower bounds. It suffices to treat the case where  \(m\geq 0\),   since for $m<0$ we have 
\begin{equation}
    H_{m,b} = H_{-m,b} + 2 \abs{m} b,
\end{equation}
which is the consequence of the identity \((\frac{m}{r} - \frac{br}{2} )^2 = (\frac{-m}{r} -\frac{br}{2})^2 - 2 \abs{m} b\).\medskip

Both the upper and lower bounds  will be found using the following trial state
\begin{equation}\label{Trial state}
    u_{m,n}(r) := r^m L^m_{n-1}\biggl(\frac{br^2}{2}\biggr)\biggl(C_1 e^{\frac{b}{4}\paren{1-r^2}}+C_2 e^{-\frac{b}{4}\paren{1-r^2}}\biggr),
\end{equation}
where \(C_1, C_2\) are constants to determine such that the trial state is in the domain of the operator and \(L^m_{n-1}\) are the associated Laguerre polynomials, 
\begin{equation}\label{Associated Laguerre polynomials}
    L^m_{n-1}(s) = \frac{e^s}{s^m n-1!} \frac{d^{n-1}}{ds^{n-1}}\paren{e^{-s} s^{n+m-1}} =\sum_{l=0}^{n-1}\frac{\paren{-1}^l}{l!}\binom{n-1+m}{n-1-l}s^l,
\end{equation}
with degree $n-1$. 

For $n=1$, the trial state was introduced in \cite{HelfferSund} with $C_1=-C_2=1$.
We refer to the Appendix \ref{AppendixA} for reasoning on why we have chosen the Laguerre polynomials. The first term in $u_{m,n}$, with coefficient $C_1$, is actually a solution of the differential equation $H_{m,b} f=(2n-1)f$, but it ceases to satisfy the Dirichlet or Neumann conditions, and that is behind adding the term with coefficient $C_2$. This line of thought traces back to Bolley and Helffer in the context of semi-classical asymptotics \cite{BolleyHelffer}.  

In order to use the trial state \(u_{m,n}\) we need to impose adequate boundary conditions.
\begin{enumerate}
    \item \textbf{Dirichlet boundary conditions:} We require that \(u(1)=0\), which gives \(L^m_{n-1}\paren{\frac{b}{2}}\paren{C_1+C_2} = 0\), and consequently
\begin{equation}\label{Dirichlet boundary condition}
    C_1=-C_2,
\end{equation}
We can then take $C_1=1$ and $C_2=-1$, as in \cite{HelfferSund}.
\item \textbf{Neumann boundary conditions:} Now we need that \(u'(1)=0\), which gives
\begin{equation}\label{Neumann boundary conditions}
    C_2 = \frac{\paren{\frac{b}{2}-m}L^m_{n-1}(\frac{b}{2}) - b (L^m_{n-1})'(\frac{b}{2})}{\paren{\frac{b}{2}+m}L^m_{n-1}(\frac{b}{2}) +  b(L^m_{n-1})'(\frac{b}{2})}C_1. 
\end{equation}
\end{enumerate}
\begin{remark}\rm 
     In the Neumann case, an important observation is that for large enough \(b\), \(C_1\) and \(C_2\) have the same sign, which will be crucial in the sign we obtain in Theorem~\ref{Main theorem}. 
    
    It is also important to notice that $C_1$ and $C_2$ depend on $b$, but  to leading order, the relation between \(C_1\) and \(C_2\) is independent  of \(b\). More precisely, as $b\to+\infty$, \eqref{Neumann boundary conditions} yields
    \[C_2 =\bigl(1+ \mathcal O(b^{-1})\Bk)\bigr)C_1,\]
    so we can take $C_1=1$ and $C_2=C_2(b)$ with $C_2(b)=1+\mathcal O(b^{-1})$.
\end{remark}

\begin{remark}\rm 
    One can also consider the Robin boundary condition $u'(1)=\gamma u(1)$, with $\gamma\in\R$; the result will be analogous to the Neumann case for fixed Robin parameter \(\gamma\), with the sub-leading term  independent of $\gamma$. In particular, the boundary condition gives
    \begin{equation}\label{Robin boundary conditions}
    C_2 = \frac{\paren{\frac{b}{2}-m+\gamma }L^m_{n-1}(\frac{b}{2}) - b (L^m_{n-1})'(\frac{b}{2})}{\paren{\frac{b}{2}+m-\gamma}L^m_{n-1}(\frac{b}{2}) +  b(L^m_{n-1})'(\frac{b}{2})}C_1, 
\end{equation}
which, as in the Neumann case, reads as \( C_2 =\bigl(1+\mathcal O(b^{-1})\bigr)C_1\) as \(b\rightarrow +\infty\).
\end{remark}

\subsection{Concluding comments}

The study of the Neumann realization of the magnetic Laplacian traces back to  works of Saint-James and de\,Gennes \cite{SaintJames, SaintJames1965EtudeDC} in the context of superconductivity. For the disc, their plot of the eigenvalue branches indicate the linear dependence on the magnetic field to leading order, consistently with the asymptotics in \eqref{eq:main}.

In contrast to the Dirichlet realization, the lowest eigenvalue for the Neumann realization $H_b^N$ does not correspond to the branch $m=0$. The study of the lowest eigenvalue for the Neumann magnetic Laplacian is then significantly harder and was the subject of a vast literature, starting with \cite{BaumanPhillipsTang} for the disc and \cite{HelfferMorame} for general domains. More recently, interest in this question emerges in the context of geometric inequalities \cite{CLPS, KLS}.

In Theorem~\ref{Main theorem}, we assumed that $m$ is fixed. Increasing $m$ in a manner dependent on $b$ can harm the result and is likely to  change the  
expression of the exponential sub-leading term as in \cite[Prop.~3.5]{FK} (and even the leading term as one observes in \cite{BaumanPhillipsTang}). That there are so many eigenvalues below each Landau level is another interesting consequence of Theorem~\ref{Main theorem}, but this was known for the lowest Landau level and  general domains \cite{FFGKS}.\medskip

The proof of Theorem~\ref{Main theorem} occupies the rest of this paper, where in Section~\ref{Section 2} we derive an eigenvalue upper bound via the min-max principle, and in Section~\ref{Section 3}, we derive the matching lower bound via the Temple inequality. In the appendix, we collect various computations and a standard comparison argument that were used in the earlier sections.

\section{Upper bound}\label{Section 2}

We denote by $\norm{\cdot}$ and $\inprod{\cdot,\cdot}$ the norm and inner product in $L^2((0,1),rdr)$. 
We will  use the Raylegh--Ritz variational formula \cite[Section 4]{Plum} which gives
\begin{equation}\label{Raylegh-Ritz Dirichlet}
     \lambda_n\paren{H^D_{m,b}}  = \min_{\substack{U\subset \mathcal{D}\paren{H^D_{m,b}} \\ \text{dim}U = n}} \max_{u\in U\setminus\{0\}} \frac{\inprod{H^D_{m,b}u,u}}{\inprod{u,u}},
\end{equation}
and 
\begin{equation}\label{Raylegh-Ritz Neumann}
     \lambda_n\paren{H^N_{m,b}}  = \min_{\substack{U\subset \mathcal{D}\paren{H^N_{m,b}} \\ \text{dim}U = n}} \max_{u\in U\setminus\{0\}} \frac{\inprod{H^N_{m,b}u,u}}{\inprod{u,u}}.
\end{equation}
\subsection{Linear Independence}
We need to check that the trial states in \eqref{Trial state} constitute a linearly independent set. We will actually prove that they are almost orthogonal, up to small errors. We start by inspecting the norms. 
\begin{lemma}\label{Lemma norm umn}
   Given integers $m\geq 0$ and $n\geq 1$,  let \(u_{m,n}\) be defined as in \eqref{Trial state}.  Then, as \(b\rightarrow +\infty\),
    \begin{equation}\label{Norm umn}
         \norm{u_{m,n}}^2 =\frac{C_1^2 \Gamma(m+n)2^m}{b^{m+1} \paren{n-1}!}e^{\frac{b}{2}} \ +\mathcal{O}\paren{b^{2n-2}},
    \end{equation}
    where $\Gamma$ is the Gamma function. 
\end{lemma}
\begin{proof}
    By \eqref{Trial state} we have
    \begin{align*}
        \norm{u_{m,n}}^2  = &\int_0^1 r^{2m+1} L^m_{n-1}\biggl(\frac{br^2}{2}\biggr)^2\biggl(C_1^2 e^{\frac{b}{2}\paren{1-r^2}} +C_2^2  e^{-\frac{b}{2}\paren{1-r^2}} +2C_1C_2\biggr) \, dr \\
                          = &\frac{2^{m}}{b^{m+1}}\int_0^{\frac{b}{2}}  s^m L^m_{n-1}\paren{s}^2 \biggl(C_1^2 e^{\frac{b}{2}-s} +C_2^2  e^{-\frac{b}{2}+s} +2C_1C_2\biggr)\, ds  ,
    \end{align*}
    where we have used the change of variables \(s= \frac{br^2}{2}\). We treat each term separately. 
    \begin{enumerate}
        \item \textbf{\(\mathbf{C_1^2}\) term: }For this term we can use the orthogonality of the associated Laguerre polynomials\footnote{See (18.3.1) in \href{http://dlmf.nist.gov/18.3.T1}{http://dlmf.nist.gov/18.3.T1}.} 
        \begin{equation}
            \int_0^{\infty} s^m  L^m_{n-1}\paren{s}^2  e^{-s} \, ds = \frac{\Gamma(m+n)}{n-1!},
        \end{equation}
         Thus, 
        \begin{multline*}
              C_1^2 \frac{2^{m}}{b^{m+1}} e^{\frac{b}{2}}\int_0^{\frac{b}{2}}  s^m L^m_{n-1}\paren{s}^2 e^{-s}\, ds = \\
              =C_1^2 \frac{2^{m}}{b^{m+1}}e^{\frac{b}{2}}\biggl( \frac{\Gamma(m+n)}{n-1!}  -  \int_{\frac{b}{2}}^{\infty}  s^m L^m_{n-1}\paren{s}^2 e^{-s}\, ds \biggr). 
        \end{multline*}
        Note that the second term is equal to
        \begin{multline*}
             \int_{\frac{b}{2}}^{\infty}  s^m L^m_{n-1}\paren{s}^2 e^{-s}\, ds \\
             = \sum_{k=0}^{n-1}\sum_{l=0}^{n-1}\frac{\paren{-1}^{l+k}}{k!l!}\binom{n-1+m}{n-1-k}\binom{n-1+m}{n-1-l} \int_{\frac{b}{2}}^{\infty}s^{m+l+k} e^{-s} \, ds,
        \end{multline*}
with $\int_{\frac{b}{2}}^{\infty}s^{m+l+k} e^{-s} \, ds=\Gamma(m+l+k+1,b/2)$ is the upper incomplete Gamma function, which satisfies\footnote{See \href{http://dlmf.nist.gov/8.2.E2}{\tt http://dlmf.nist.gov/8.2.E2} and \href{http://dlmf.nist.gov/8.4.E8}{\tt http://dlmf.nist.gov/8.4.E8}.} $\Gamma(x+1,z)=x!e^{-x}\sum_{k=0}^xx^k/k!$. Consequently,
\begin{multline*}
 e^{\frac{b}{2}}\int_{\frac{b}{2}}^{\infty}  s^m L^m_{n-1}\paren{s}^2 e^{-s}\, ds \\
     = \sum_{k=0}^{n-1}\sum_{l=0}^{n-1}\frac{\paren{-1}^{l+k}}{k!l!}\binom{n-1+m}{n-1-k}\binom{n-1+m}{n-1-l} \paren{m+l+k}! \sum_{j=0}^{m+l+k} \frac{b^j}{2^k j!}\\
     =\mathcal O\paren{b^{m+2n-2}}.
\end{multline*}
Hence, 
        \begin{equation}\label{C1^2 term}
            C_1^2 \frac{2^{m}}{b^{m+1}} e^{\frac{b}{2}}\int_0^{\frac{b}{2}}  s^m L^m_{n-1}\paren{s}^2 e^{-s}\, ds = \frac{C_1^2  \Gamma(m+n)2^m}{b^{m+1} \paren{n-1}!}e^{\frac{b}{2}} + \mathcal{O}\paren{b^{2n-3}}.
        \end{equation}
        \item \textbf{\(\mathbf{C_2^2}\) term: } The computations involve an integral reminiscent of the   lower incomplete Gamma function 
        \[\int_0^{z} e^s s^{x}ds=(-1)^xx!\biggl(e^{z}\sum_{t=0}^x\frac{(-z)^t}{t!}-1 \biggr).\] 
        The above formula can be verified by an iteration of integration by parts. 
        Consequently,\Bk
        \begin{equation}\label{Integral smlk exps}
        \begin{aligned}
        e^{-\frac{b}{2}}\int_0^{\frac{b}{2}}  e^{s}s^{m+l+k}\, ds 
        &=(-1)^{m+l+k}\paren{m+l+k}!  \biggl( \sum_{t=0}^{m+l+k} \frac{(-b)^t}{2^t t!}- e^{-\frac{b}{2}}\biggl)\\
        &=\mathcal O\paren{b^{m+l+k}}.
        \end{aligned}
    \end{equation}
    Then, 
        \begin{multline*}
                \frac{C_2^2  2^{m}}{b^{m+1}} e^{-\frac{b}{2}}\int_0^{\frac{b}{2}}  s^m L^m_{n-1}\paren{s}^2  e^{s} \, ds \\
                =   \frac{C_2^2  2^{m}}{b^{m+1}} e^{-\frac{b}{2}} \sum_{k=0}^{n-1}\sum_{l=0}^{n-1}\frac{\paren{-1}^{l+k}}{k!l!}\binom{n-1+m}{n-1-k}\binom{n-1+m}{n-1-l} \int_0^{\frac{b}{2}} s^{m+l+k} e^s \, ds, 
        \end{multline*}
        and by \eqref{Integral smlk exps}, 
        \[ \frac{C_2^2  2^{m}}{b^{m+1}} e^{-\frac{b}{2}}\int_0^{\frac{b}{2}}  s^m L^m_{n-1}\paren{s}^2  e^{s} \, ds=\mathcal{O}\paren{b^{2n-3}}.\]
        \item \textbf{\(\mathbf{C_1C_2}\) term: } 
    The last term gives 
        \begin{multline*}
            \frac{C_1C_2 2^{m+1}}{b^{m+1}}\int_0^{\frac{b}{2}}  s^m L^m_{n-1}\paren{s}^2  \,  ds \\
            =   \frac{C_1C_2 2^{m+1}}{b^{m+1}} \sum_{k=0}^{n-1}\sum_{l=0}^{n-1}\frac{\paren{-1}^{l+k}}{k!l!}\binom{n-1+m}{n-1-k}\binom{n-1+m}{n-1-l} \int_0^{\frac{b}{2}} s^{m+l+k} \, ds\\
            = \mathcal{O}\paren{b^{2n-2}},
        \end{multline*}
        where in the last step, we used that 
        \[\int_0^{\frac{b}{2}}   s^{m+l+k}  \,  ds =\frac{1}{\paren{m+l+k+1}}\biggl(\frac{b}{2}\biggr)^{m+l+k+1}.\qedhere\] 
    \end{enumerate}
    \end{proof}
Next we check the linear independence of our trial states. 
\begin{lemma}\label{Lemma 1}
    Given integers \(m\geq 0\) and $n\geq 1$, there exists $b_0>0$ such that, for all $b\geq b_0$, the set \(\{u_{m,1}, \ldots, u_{m,n}\}\) spans  a subspace of dimension \(n\), where
     \(u_{m,k}\) is introduced in \eqref{Trial state}. 

     Furthermore, for $i\not=j$, we have $\langle u_{m,i},u_{m,j}\rangle=\mathcal O(b^{i+j-3})$.
\end{lemma}
\begin{proof}
    Denote by \(C_{1,n}\) and \(C_{2,n}\) the constants \(C_1\) and \(C_2\) corresponding to \(u_{m,n}\) in \eqref{Trial state}. If \(1\leq i<j\leq n\)
\begin{align*}
        \inprod{u_{m,i}, u_{m,j}} =& \int_0^1 r^{2m+1} L^m_{i-1}\biggl(\frac{br^2}{2}\biggr)L^m_{j-1}\biggl(\frac{br^2}{2}\biggr)\\
                                    &\biggl(C_{1,i}C_{1,j} e^{\frac{b}{2}\paren{1-r^2}} +C_{2,i}C_{2,j} e^{-\frac{b}{2}\paren{1-r^2}} + C_{1,i}C_{2,j} + C_{2,i}C_{1,j}\biggr) \, dr \\
                                =& \frac{2^m}{b^{m+1}} \int_0^{\frac{b}{2}} s^m  L^m_{i-1}\paren{s} L^m_{j-1}\paren{s}  \\
                                    & \biggl(C_{1,i}C_{1,j} e^{\frac{b}{2}-s}+C_{2,i}C_{2,j} e^{s-\frac{b}{2}} + C_{1,i}C_{2,j} + C_{2,i}C_{1,j}\biggr) \, ds,
\end{align*}
where we have used the change of variables \(s= \frac{br^2}{2}\). We can estimate
\begin{equation*}
    \abs{C_{1,i}C_{2,j}} \leq \frac{C_{1,i}^2}{2} e^{\frac{b}{2}-s} + \frac{C_{1,j}^2}{2}  e^{s-\frac{b}{2}} \text{ and }  \abs{C_{2,i}C_{1,j}} \leq \frac{C_{1,j}^2}{2} e^{\frac{b}{2}-s} + \frac{C_{2,i}^2}{2}  e^{s-\frac{b}{2}} . 
\end{equation*}
Hence, 
\begin{multline}\label{eq:inn-prod}
      \abs{\inprod{u_{m,i}, u_{m,j}} }\leq  \frac{2^mD_1}{b^{m+1}} e^{\frac{b}{2}} \bigabs{\int_0^{\frac{b}{2}}  s^m  L^m_{i-1}\paren{s} L^m_{j-1}\paren{s}   e^{-s} \, ds} \\
                                + \frac{2^mD_2}{b^{m+1}} e^{-\frac{b}{2}}\bigabs{\int_0^{\frac{b}{2}}  s^m  L^m_{i-1}\paren{s} L^m_{j-1}\paren{s}   e^{s} \, ds } ,
\end{multline}
where \(D_1 = \abs{C_{1,i}C_{1,j}} + \frac{C_{1,i}^2}{2} + \frac{C_{1,j}^2}{2}\) and \(D_2 =\abs{C_{2,i}C_{2,j}} + \frac{C_{2,i}^2}{2} + \frac{C_{2,j}^2}{2} \).  Using that the Laguerre polynomials are orthogonal\footnote{See (18.3.1) in \href{http://dlmf.nist.gov/18.3.T1}{\tt http://dlmf.nist.gov/18.3.T1}} over \([0, +\infty)\) with respect to the weight \(s^m e^{-s}\)  we have
\begin{align}\label{Orthogonality Laguerre}
     \int_0^{\infty} s^m  L^m_{i-1}\paren{s} L^m_{j-1}\paren{s}   e^{-s} \, ds = & \frac{\Gamma(m+i)}{i-1!} \delta_{(i-1)(j-1)} = 0,
\end{align}
since \(i \neq j\). Then, we can rewrite, using the formula for the incomplete Gamma function\footnote{See (8.4.8) in \href{http://dlmf.nist.gov/8.4.E8}{\tt http://dlmf.nist.gov/8.4.E8}.}\Bk
\begin{equation}\label{2.3}
\begin{split}
    e^{\frac{b}{2}} & \int_0^{\frac{b}{2}} s^m  L^m_{i-1}\paren{s} L^m_{j-1}\paren{s}   e^{-s} \, ds =  - e^{\frac{b}{2}} \int_{\frac{b}{2}}^\infty s^m  L^m_{i-1}\paren{s} L^m_{j-1}\paren{s}   e^{-s} \, ds\\
    & = - e^{\frac{b}{2}}\sum_{k=0}^{i-1}\sum_{l=0}^{j-1} \frac{\paren{-1}^{l+k}}{l!k!}\binom{i-1+m}{i-1-k} \binom{j-1+m}{j-1-l} \underbrace{\int_{\frac{b}{2}}^\infty s^{m+k+l} e^{-s}\, ds }_{=\Gamma(m+1, \frac{b}{2})}\\
    &=  - \sum_{k=0}^{i-1}\sum_{l=0}^{j-1} \frac{\paren{-1}^{l+k}}{l!k!}\binom{i-1+m}{i-1-k} \binom{j-1+m}{j-1-l} \paren{m+k+l}! \sum_{t=0}^{m+k+l} \frac{b^t}{2^t t!}\\
    &=\mathcal O(b^{m+i+j-2}).
\end{split}
\end{equation}
On the other hand, 
\begin{equation}\label{2.4}
\begin{split}
     &e^{-\frac{b}{2}}\int_0^{\frac{b}{2}}  s^m  L^m_{i-1}\paren{s} L^m_{j-1}\paren{s}   e^{s} \, ds   \\
     = &\ e^{-\frac{b}{2}}\sum_{k=0}^{i-1}\sum_{l=0}^{j-1} \frac{\paren{-1}^{l+k}}{l!k!}\binom{i-1+m}{i-1-k} \binom{j-1+m}{j-1-l} \int_0^{\frac{b}{2}} s^{m+k+l} e^{s}\, ds =  \\
    = & (-1)^m \sum_{k=0}^{i-1}\sum_{l=0}^{j-1} \frac{1}{l!k!}\binom{i-1+m}{i-1-k} \binom{j-1+m}{j-1-l}\paren{m+k+l}!  \biggl( \sum_{t=0}^{m+k+l} \frac{(-b)^t}{2^t t!}- e^{-\frac{b}{2}}\biggl)\\=& \mathcal O(b^{m+i+j-2}).
\end{split}
\end{equation}
Inserting \eqref{2.3} and \eqref{2.4} into \eqref{eq:inn-prod}, we get
\[\langle u_{m,i},u_{m,j}\rangle=\mathcal O(b^{i+j-3}).\]
We computed \(\norm{u_{m,i}}^2\) and \(\norm{u_{m,j}}^2\) in Lemma \ref{Norm umn}, and we found that they are exponentially large. The terms  in \eqref{2.3} and \eqref{2.4} do not compensate this exponential decay, which means that 
\begin{equation*}
     \bigabs{\biginprod{\frac{u_{m,i}}{\norm{u_{m,i}}},\frac{u_{m,j}}{\norm{u_{m,j}}}}} = \mathcal{O}\paren{b^{m+i+j-2}e^{-\frac{b}{2}}}. \qedhere
\end{equation*}
\end{proof}

\subsection{The Rayleigh--Ritz quotient}
Now we can start to compute the numerator and denominator of Rayleigh--Ritz quotient. 
\begin{lemma}\label{Lemma numerator}
Given integers $m\geq 0$ and $n\geq 1$, it holds as $b\to+\infty$,
    \begin{equation}\label{Equation lemma numerator}
        \inprod{H_{m,b} u_{m,n}, u_{m,n}} = \paren{2n-1}b \norm{u_{m,n}}^2 -   C_1C_2  \frac{b^{2n-1}}{2^{2n-2}\paren{(n-1)!}^2} + \mathcal{O}\paren{b^{2n-2}} ,
\end{equation}
where \(u_{m,n}\) was introduced in \eqref{Trial state}.
\end{lemma}
\begin{proof}
    First we need to compute \(H_{m,b}u_{m,n}\) where \(H_{m,b}\) was defined in \eqref{Hm,b}. By \eqref{eq:DE-f} and \eqref{eq:DE-g},  we know that 
\begin{equation}\label{Hm un}
        H_{m,b} u_{m,n} = \paren{2n-1}b u_{m,n}
                          + R_{m,n}, 
\end{equation}
where
\begin{equation}\label{Rmb}
 R_{m,n}=-2 b C_2   \Bk r^m  e^{-\frac{b}{4}\paren{1-r^2}} \biggl( br^2 \paren{L^m_{n-1}}'\Bigr(\frac{br^2}{2}\Bigl) + \paren{m+1} L^m_{n-1}\Bigl(\frac{br^2}{2}\Bigr) \biggr).
\end{equation}
It is important to notice that both \(L_{n-1}^m\) and its derivative have the same sign for \(b\) large enough (the leading order terms have the same sign). We can finally compute
\begin{equation}\label{inprod{H_{m,b} u_n, u_n}}
        \inprod{H_{m,b} u_{m,n}, u_{m,n}} = \paren{2n-1}b \norm{u_n}^2+ \mathrm{I}+\mathrm{II},
\end{equation}
where
\[\begin{gathered}
    \mathrm{I}= C_1\int_0^1 r^{m+1}e^{-\frac{b}{4}\paren{1+r^2}}  
         L^m_{n-1}\Bigl(\frac{br^2}{2}\Bigr) R_{m,n}(r)\, dr,\\
    \mathrm{II}=C_2\int_0^1 r^{m+1} 
         L^m_{n-1}\Bigl(\frac{br^2}{2}\Bigr)R_{m,n}(r)\, dr.
\end{gathered}\]
As we will see, the main contribution to the remainder term in \eqref{inprod{H_{m,b} u_n, u_n}} comes from the term  $\mathrm{I}$ since the term $\mathrm{II}$ is going to be of lower order in \(b\). Next, we compute these terms in more detail. 
\begin{enumerate}
    \item  \textbf{The term \(\mathrm{I}\)\,:} We need to estimate the size of
    \begin{equation*}
       2b \int_0^1 r^{2m+1}L^m_{n-1}\Bigl(\frac{br^2}{2}\Bigr)  \biggl( br^2 \paren{L^m_{n-1}}'\Bigl(\frac{br^2}{2}\Bigr) + \paren{m+1} L^m_{n-1}\Bigl(\frac{br^2}{2}\Bigr) \biggr)\, dr.
    \end{equation*}
    In this case, we have no exponential prefactor \(e^{\pm s}\), so it is not as beneficial to introduce the change of variables \(s= \frac{br^2}{2}\). Moreover, the integral with respect to \(r\) will not produce any contribution in \(b\), this means that we can cut the Laguerre polynomials to their highest order. The rest of the terms will contribute to at most the same order as the highest multiplied by \(b^{-1}\). 
    
    Note that \(\paren{L^m_{n-1}}'(\frac{br^2}{2}) \) is multiplied by \(br^2\), so its higher order term with respect to \(b\) will also contribute. Then, using \eqref{Associated Laguerre polynomials} and
    \begin{equation}\label{L^m_{n-1} derivative}
    \begin{split}
        \paren{L^m_{n-1}}'(\frac{br^2}{2}) =& \sum_{l=0}^{n-2} \frac{\paren{-1}^{l+1}}{l!} \binom{n-1+m}{n-2-l}\biggl(\frac{br^2}{2}\biggr)^{l}\\
                                            =& \sum_{l=1}^{n-1} \frac{\paren{-1}^{l}}{(l-1)!} \binom{n-1+m}{n-1-l}\biggl(\frac{br^2}{2}\biggr)^{l-1}, 
    \end{split}     
    \end{equation}
    we know that the highest order term of \[  L^m_{n-1}\Bigl(\frac{br^2}{2}\Bigr)  \biggl( br^2 \paren{L^m_{n-1}}'\Bigl(\frac{br^2}{2}\Bigr) + \paren{m+1} L^m_{n-1}\Bigl(\frac{br^2}{2}\Bigr) \biggr)  \] in \(b\) is given by 
    \begin{multline*}
        \frac{b^{2(n-1)}r^{4(n-1)}}{2^{2n-3} (n-1)!(n-2)!} + \frac{\paren{m+1}b^{2(n-1)}r^{4(n-1)}}{2^{2n-2} \paren{(n-1)!}^2}\\
        =\frac{b^{2(n-1)}r^{4(n-1)}}{2^{2n-2}\paren{(n-1)!}^2}(2n+ m-1 ),
    \end{multline*}
    implying that the term \(\mathrm{I}\) is given by
    \begin{multline}\label{Remainder term}
    - C_1C_2  \frac{b^{2n-1}(2n+ m-1 )}{2^{2n-3}\paren{(n-1)!}^2} \int_0^1r^{2m + 4n -3}  \, dr + \mathcal{O}\paren{b^{2n-2}} \\
    = -   C_1C_2  \frac{b^{2n-1}}{2^{2n-2}\paren{(n-1)!}^2} + \mathcal{O}\paren{b^{2n-2}}  . 
    \end{multline}

	\item \textbf{The term \(\mathrm{II}\)\,:} In this case, the integrand in the term $\mathrm{II}$ is
 \[ -2C_2^2b r^{2m+1}  e^{-\frac{b}{4}\paren{1-r^2}} L^m_{n-1}\Bigl(\frac{br^2}{2}\Bigr)\biggl( br^2 \paren{L^m_{n-1}}'\Bigr(\frac{br^2}{2}\Bigl) + \paren{m+1} L^m_{n-1}\Bigl(\frac{br^2}{2}\Bigr) \biggr),\]
 and we have an exponential prefactor, so it is useful to consider the change of variables \(s= \frac{br^2}{2}\), and use \eqref{Associated Laguerre polynomials} and \eqref{L^m_{n-1} derivative} to get
	\begin{multline*}
		\sum_{l=0}^{n-1} \sum_{k=1}^{n-1} \frac{\paren{-1}^{l+k}}{l!k!}\binom{n-1+m}{n-1-l}\binom{n-1+m}{n-1-k}(2 k + m+1) \int_0^{\frac{b}{2}}  e^{s}s^{m+l+k}\, ds  \\
          + \sum_{l=0}^{n-1}  \frac{(-1)^l}{l!} \binom{n-1+m}{n-1-l}\binom{n-1+m}{n-1}\paren{m+1}\int_0^{\frac{b}{2}} e^s s^{m+l}\, ds
	\end{multline*}
	multiplied by \(2^{m+1}C_2^2e^{-\frac{b}{2}}b^{-m} \). Thanks to \eqref{Integral smlk exps}, 
    we see that the term with the biggest power in \(b\) is when \(l=k=n-1\) giving a power of \(b^{m+2n-2}\). This means that the contribution is of order \(\mathcal{O}\paren{b^{2n-2}}\) since we are multiplying the terms above by  \(2^{m+1}C_2^2e^{-\frac{b}{2}}b^{-m}\). 
    \end{enumerate}
\end{proof}
\begin{remark}\label{rem:mixed-terms}\rm 
Similar computations to the ones in the proof of Lemma~\ref{Lemma numerator} yield  for $i\not=j$ and $b\to+\infty$,
\[\langle H_{m,b}u_{m,i},u_{m,j}\rangle=\mathcal O(b^{i+j-3}).\]
\end{remark}

    Combining Lemmas~\ref{Lemma norm umn} and \ref{Lemma numerator},  we get the following result. 
    \begin{prop}
    If \(b\rightarrow +\infty\) and \(m\geq 0\), then 
    \begin{equation}
        \lambda_n\paren{H^D_{m,b}} \leq (2n-1) b  +e^{-\frac{b}{2}} \bigg( \frac{ b^{2n+m}  }{  \paren{n-1}!\paren{m+n-1}!2^{2(n-1)+m}} + \mathcal{O}\paren{b^{2n+m{ -1}}}\biggr), 
    \end{equation}
    and 
    \begin{equation}
        \lambda_n\paren{H^N_{m,b}}\leq  (2n-1) b  - e^{-\frac{b}{2}} \bigg( \frac{ b^{2n+m}  }{  \paren{n-1}!\paren{m+n-1}!2^{2(n-1)+m}} + \mathcal{O}\paren{b^{2n+m{ -1}}}\biggr). 
    \end{equation}
    \end{prop}
    \begin{proof}
        We choose $M=\mathrm{Span}\{u_{m,k}\colon 1\leq k\leq n\}$ in \eqref{Raylegh-Ritz Dirichlet} or \eqref{Raylegh-Ritz Neumann}. The result then follows by using  Lemmas~\ref{Lemma norm umn} and \ref{Lemma numerator}, and by using that $C_1C_2=-1$ for Dirichlet, and $C_1C_2=1+\mathcal O(b^{-1})$ for Neumann.
    \end{proof}

\section{Lower bound}\label{Section 3}
\subsection{The Temple quotient}

We recall a result from  \cite[Thm.~3]{Plum}, also known as Kato's theorem (see  \cite[Eq. 10]{Kato}), which states a lower bound on the $n$-th eigenvalues, conditional that there are  \emph{a priori} upper and lower bounds on the $(n-1)$-th and $(n+1)$-th eigenvalue respectively.

More precisely, let $\#\in\{D,N\}$ and suppose that \(\nu_{n+1} >0\) is a lower bound for the \(\paren{n+1}\)-th eigenvalue \(\lambda_{n+1}\paren{H^\#_{m,b}}\) and that $\mu_{n-1}$ is a lower bound for the \(\paren{n-1}\)-th eigenvalue. If furthermore, there is a trial state \(u \in \mathcal{D}(H_{m, b}^\#) \setminus \{0\}\), 
satisfying \(\mu_{n-1}<\inprod{H_{m,b} u, u}/\inprod{u,  u} < \nu_{n+1} \), then we get the following bounds  from below
\begin{equation}\label{Temple-Rayleigh bracketing}
       \frac{\nu_{n+1} \inprod{H_{m,b} u, u}-\inprod{H_{m,b} u, H_{m,b}u}}{\nu_{n+1}\norm{u}^2 - \inprod{H_{m,b} u , u}} \leq \lambda_{n-1}\paren{H_{m,b}^\#},
\end{equation}
for \(n \in \mathbb{N}\).
The quantity on the left hand side is usually called the Temple quotient.

We will use again \(u_{m,n}\) defined in \eqref{Trial state} as trial state, and thanks to Proposition~\ref{Proposition 1}, we will take $\mu_{n-1}=(2n-3)b+1$. The condition $\mu_{n-1}<\inprod{H_{m,b} u, u}/\inprod{u,  u}$ holds for $u=u_{m,n}$ after Lemmas~\ref{Lemma norm umn} and \ref{Lemma 1}. We then just need to select $\nu_{n+1}$ in an appropriate manner and calculate the Temple quotient.\Bk
\subsection{Rough lower bound}
Before starting the computations of the Temple quotient we need to find suitable
\(\nu_{n+1}\) for each \(n\in \mathbb{N}\). 
By a standard comparison argument that we recall in Appendix~\ref{app:B}, we prove the following.
\begin{prop}\label{Proposition 1}
    For any integer \(m \geq 0\), there exist $C,b_0>0$ such that,  for \(n=1, 2,3, \ldots\) and $b\geq b_0$, we have
    \begin{equation}
       \lambda_n \paren{H^D_{m,b}} \geq \lambda_n \paren{H^N_{m,b}} \geq (2n-1)b - C. 
    \end{equation}
\end{prop}
We can then take $\nu_{n+1}=(2n+1)b-C$ on the left hand side of \eqref{Temple-Rayleigh bracketing}, both for the Dirichlet and Neumann realizations.

\subsection{Computation of the Temple quotient}
In Proposition \ref{Proposition 1}, we showed that we can choose \(\nu_{n+1} =(2n+1)b - C \). A closer look to the Temple quotient tells us that only \(\inprod{H_{m,b}u_{m,n},H_{m,b}u_{m,n}} \) needs to be computed since the other terms were computed for \(u_{m,n}\) in Lemma \ref{Lemma norm umn} and Lemma \ref{Lemma numerator}. 
\begin{lemma}\label{Lemma norm Hmb umn}
     Let \(u_{m,n}\) be defined as in \eqref{Trial state}, then
    \begin{equation}\label{Norm Hmb umn}
         \inprod{H_{m,b}u_{m,n},H_{m,b}u_{m,n}} = \paren{2n-1}b  \inprod{H_{m,b} u_{m,n},u_{m,n}} + \inprod{H_{m,b} u_{m,n},R_{m,n}} ,
    \end{equation}
    and
     \begin{equation}\label{3.14}
         \inprod{H_{m,b}u_{m,n},u_{m,n}} = \paren{2n-1}b \norm{u_{m,n}}^2+ \inprod{ u_{m,n},R_{m,n}}, 
    \end{equation}
    where $R_{m,n}$ was introduced in \eqref{Rmb}.
\end{lemma}
\begin{proof}
This is a direct consequence the following identity established in \eqref{Hm un}. 
\end{proof}
We will need the following lemma. 
\begin{lemma}\label{Lemma 8}
    Let \(u_{m,n}\) be  as in \eqref{Trial state} and \(R_{m,n}\) as in \eqref{Rmb}, then
    \begin{multline}
          \frac{\inprod{R_{m,n}, u_{m,n}\paren{2b -C}-R_{m,n} }}{\inprod{ u_{m,n}, u_{m,n}\paren{2b -C}-R_{m,n} }} \\
          = \biggl(-\frac{    C_2 b^{2n+m} }{C_12^{2(n-1)+m}(n-1)! \Gamma(m+n)}  + \mathcal{O}\paren{b^{2n+m{-1}}}\biggr)e^{-\frac{b}{2}},
    \end{multline}
    as \(b\rightarrow +\infty\).
\end{lemma}
\begin{proof}
We start by computing \(\inprod{R_{m,n}, u_{m,n}}\) as 
    \begin{multline*}
        \inprod{R_{m,n}, u_{m,n}} = -2b\int_0^1 r^{2m+1} \paren{C_1C_2 + C_2^2  e^{-\frac{b}{2}\paren{1-r^2}}}\\
        \biggl( br^2 \paren{L^m_{n-1}}'\Bigl(\frac{br^2}{2}\Bigr) + \paren{m+1} L^m_{n-1}\Bigl(\frac{br^2}{2}\Bigr) \biggr) \, dr,
    \end{multline*}
    which is the same term as the second summand in \eqref{inprod{H_{m,b} u_n, u_n}}. Thus, using the results from Lemma \ref{Lemma numerator},
    \begin{equation}\label{3.17}
       \inprod{R_{m,n}, u_{m,n}} =  -   C_1C_2  \frac{b^{2n-1}}{2^{2n-2}\paren{(n-1)!}^2} + \mathcal{O}\paren{b^{2n-2}}. 
    \end{equation}
    On the other hand, by Lemma \ref{Lemma norm umn}
    \begin{equation}\label{3.18}
         \norm{u_{m,n}}^2 =\frac{C_1^2 \Gamma(m+n)2^m}{b^{m+1} \paren{n-1}!}e^{\frac{b}{2}} \ +\mathcal{O}\paren{b^{2n-2}},
    \end{equation}
    as \(b\rightarrow +\infty\). Then, we only need to compute
    \begin{multline*}
         \norm{R_{m,n}}^2 \\
          = 4C_2^2 b^2 \int_0^1 r^{2m+1}  e^{-\frac{b}{2}\paren{1-r^2}}\biggl( br^2 \paren{L^m_{n-1}}'\Bigl(\frac{br^2}{2}\Bigr) + \paren{m+1} L^m_{n-1}\Bigl(\frac{br^2}{2}\Bigr) \biggr)^2 \, dr\\
                         = \frac{2^{m+3}C_2^2}{b^{m-1}}e^{-\frac{b}{2}} \int_0^{\frac{b}{2}} s^m  e^{s}\biggl( 2s \paren{L^m_{n-1}}'(s) + \paren{m+1} L^m_{n-1}(s) \biggr)^2 \, ds,
    \end{multline*}
    where we have used the change of variables \(s= \frac{br^2}{2}\). This integral is no much different than the one considered in the \(C_2^2 \) term in the proof of Lemma \ref{Lemma numerator}. In fact, recall that we had in \eqref{Integral smlk exps} 
    \begin{equation*}
        \int_0^{\frac{b}{2}}  e^{s}s^{m+l+k}\, ds = \paren{m+k+l}!  \biggl( \sum_{t=0}^{m+k+l}e^{\frac{b}{2}} \frac{(-b)^t}{2^t t!}- 1\biggl),
    \end{equation*}
    where \(l,k \) were positive integers. This means that the term coming from the integral is of order \(\mathcal{O}\paren{b^{m+2n-2}}\) implying that \( \norm{R_{m,n}}^2 = \mathcal{O}\paren{b^{2n-1}} \).  Combining this with \eqref{3.17} and \eqref{3.18} finishes the proof. An important remark is that even though \(\norm{R_{m,n}}^2\) has the same order as the main term in \eqref{3.17}, we are considering \(\inprod{R_{m,n}, 2bu_{m,n}}\) which adds a factor \(b\). Because of this, \(\norm{R_{m,n}}^2\) is a  lower order term.
\end{proof}
We are ready to compute the lower bound. 
\begin{prop}\label{Proposition lower bound}
   Given integers $n\geq 1$ and $m\geq0$, then as \(b\rightarrow +\infty\), 
    \begin{equation}
        \lambda_n\paren{H^D_{m,b}} \geq (2n-1) b  +e^{-\frac{b}{2}} \bigg( \frac{ b^{2n+m}  }{  \paren{n-1}!\paren{m+n-1}!2^{2(n-1)+m}} + \mathcal{O}\paren{b^{2n+m{ -1}}}\biggr), 
    \end{equation}
    and 
    \begin{equation}
        \lambda_n\paren{H^N_{m,b}}\geq  (2n-1) b  - e^{-\frac{b}{2}} \bigg( \frac{ b^{2n+m}  }{  \paren{n-1}!\paren{m+n-1}!2^{2(n-1)+m}} + \mathcal{O}\paren{b^{2n+m{ -1}}}\biggr). 
    \end{equation}
\end{prop}
\begin{proof}
  Let \(\nu_{n+1}= (2n+1)b -C\). Thanks to the  identities in Lemma~\ref{Lemma norm Hmb umn},
 \begin{multline*}
       \frac{\nu_{n+1} \inprod{H_{m,b} u_{m,n}, u_{m,n}}-\inprod{H_{m,b} u_{m,n}, H_{m,b}u_{m,n}}}{\nu_{n+1}\norm{u_{m,n}}^2 - \inprod{H_{m,b} u_{m,n} , u_{m,n}}} \\
       = (2n-1) b +\frac{\inprod{R_{m,n}, u_{m,n}\paren{2b -C}-R_{m,n} }}{\inprod{ u_{m,n}, u_{m,n}\paren{2b -C}-R_{m,n} }}.
  \end{multline*}
  The result follows from Lemma \ref{Lemma 8} and \eqref{Temple-Rayleigh bracketing}. 
\end{proof}

\subsection*{Acknowledgements}{\small
The authors would like to thank Mikael Persson Sundqvist for  suggesting this problem and for the useful  discussions. AK is partially supported by  CUHK-SZ grant no. UDF01003322, and project no. UF02003322.} 
\appendix

\section{Associated Laguerre polynomials}\label{AppendixA}
To build our trial state, we need to get some intuition by finding solutions for \(H_{m,b} f = (2n-1)b f\), with \(n\in \mathbb{N}\). The  function $\psi(r)=r^me^{-br^2/4}$ satisfies $\psi'(r)= (m/r-br/2)\psi(r)$, so it is straightforward to check that $H_{m,b}\psi=b\psi$.
Let \(f(r) = r^m e^{-\frac{br^2}{4}} v(r)=\psi(r)v(r)\). Then, 
\begin{equation}\label{eq:DE-f}
H_{m,b} f =(2n-1)b f
\end{equation}
holds if and only the function $v$ satisfies
\begin{equation}\label{eq:DE-v}
    -rv''(r) +\paren{br^2-2m- 1 }v'(r) -(2n-2)brv(r) =0.
\end{equation}
Introducing the change of variables \(s= \frac{br^2}{2}\), and the function \(w(s) = v(r)\),  we get that \eqref{eq:DE-v} is equivalent to 
\begin{equation}\label{Laguerre equation}
    sw''(s) + \paren{m+1 - s}w'(s) +(n-1)w(s) = 0.
\end{equation}
For integers $m\geq0$ and $n\geq 1$, a solution of the differential equation in \eqref{Laguerre equation} is the  Laguerre polynomial of degree $n$,  
\[L^m_{n-1}(s) = \frac{e^s}{s^m n-1!} \frac{d^{n-1}}{ds^{n-1}}\paren{e^{-s} s^{n+m-1}} =\sum_{l=0}^{n-1}\frac{\paren{-1}^l}{l!}\binom{n-1+m}{n-1-l}s^l.\]

We set $v(r)=L^m_{n-1}(br^2/2)$ and recall that it satisfies \eqref{eq:DE-v}. We introduce the functions $\phi(r)=r^m e^{br^2/4}$ and $g(r)=\phi(r)v(r)$. Firstly, using that $\phi'(r)=(m/r+br/2)\phi(r)$, it is straightforward to  verify that
\[H_{m,b}\phi=-(2m+1)b\phi.\]
Secondly, we notice that $g$ satisfies
\[H_{m,b}g=-(2m+1)bg+r^{-1}\bigl(-rv''+(-br^2-2m-1)v' \bigr)\phi.\]
Thanks to \eqref{eq:DE-v}, we get
\begin{equation}\label{eq:DE-g}
H_{m,b}g=(2n-2m-3)bg-2brv'(r)\phi(r).
\end{equation}

\section{A comparison argument}\label{app:B}

Suppose that $m\geq 0$ is fixed and \(b\geq 2\). The operator $H_{m,b}^N$ is the self-adjoint operator associated to the closed and positive  quadratic  form
\begin{equation}
    q_{m,b} (u) = \int_0^1 \biggl(\abs{u'(r)}^2  +\biggl(\frac{m}{r}-\frac{br}{2}\biggr)^2 \abs{u(r)}^2 \biggr) \,r dr 
\end{equation}
with domain \(\mathcal{D}\paren{{q}_{m,b}} = \{u \in L^2\paren{\paren{0,1}, rdr}\colon u',mu/r\in L^2\paren{\paren{0,1}, rdr}\}\).
\begin{lemma}\label{Lemma robin boundary}
   If $b\geq 16m$, \(u\in \mathcal{D}\paren{q_{m,b}}\) and $\mathrm{supp}\,u\subset(\frac12,1)$, then 
   \[q_{m,b}(u)\geq \frac{b^2}{64}\int_0^1|u|^2rdr.\] 
\end{lemma}
\begin{proof}
    For \(\frac{1}{2}\leq r<1\), we have \(br/2-m/r\geq b/4-2m\).
\end{proof}

We consider  two self-adjoint operators in $L^2((0,3/4),r\, dr)$ and  $L^2((1/2,1),r\,dr)$ respectively. The first operator is  defined as
\begin{equation}\label{eq:app-def-hat-H}
    \begin{aligned}
        \hat{H}_{m,b} &= -\frac{d^2}{dr^2}-\frac1r\frac{d}{dr} + \Bigl(\frac{m}{r}-\frac{br}2\Bigr)^2,\\
    \mathcal{D}\paren{\hat{H}_{m,b}}& = \{u\colon u,mu/r,u',  \hat{H}_{m,b}  u\in L^2((0, 3/4), rdr),\\
    &\qquad u(3/4)=0,~\lim_{r\to0^+}\frac{u(r)}{\ln r}=0\mbox{ for }m=0\},
    \end{aligned}
\end{equation} 
and associated to the quadratic form 
\begin{equation}\label{eq:app-def-hat-q}
\begin{aligned}
     \hat{q}_{m,b}(u)& = \int_0^{\frac{3}{4}} \biggl(  \abs{u'(r)}^2  +\biggl(\frac{m}{r}-\frac{br}{2}\biggr)^2 \abs{u(r)}^2 \biggr) \, rdr,\\
     \mathcal{D}\paren{\hat{q}_{m,b}}&= \{u\colon u,mu/r,u'\in L^2((0, 3/4), rdr),~ u(3/4)=0\}.
\end{aligned}
\end{equation}
The second operator is defined as
\begin{equation}\label{eq:app-def-tilde-H}
    \begin{aligned}
        \tilde{H}_{m,b} &= -\frac{d^2}{dr^2}-\frac1r\frac{d}{dr} + \Bigl(\frac{m}{r}-\frac{br}2\Bigr)^2,\\
    \mathcal{D}\paren{\tilde{H}_{m,b}}& = \{u\colon u,u/r,u',  \hat{H}_{m,b}  u\in L^2((1/2,1), rdr),~ u(1/2)=0\},
    \end{aligned}
\end{equation} 
and associated to the quadratic form 
\begin{equation}\label{eq:app-def-tilde-q}
\begin{aligned}
     \tilde{q}_{m,b}(u)& = \int_{\frac{1}{2}}^{1} \biggl( \abs{u'(r)}^2  +\biggl(\frac{m}{r}-\frac{br}{2}\biggr)^2 \abs{u(r)}^2 \biggr) \, rdr,\\
     \mathcal{D}\paren{\tilde{q}_{m,b}}&= \{u\colon u,u'\in L^2((1/2, +\infty), rdr),~ u(1/2)=0\}.
\end{aligned}
\end{equation}
The operators \( \hat{H}_{m,b} \) and \(\tilde{H}_{m,b} \) have discrete spectra, so let \(\lambda_n\paren{\hat{H}_{m,b}}\) and  \(\lambda_n\paren{\tilde{H}_{m,b}}\) be the $n$-th eigenvalue of \(\hat{H}_{m,h,0}\) and \(\hat{H}_{m,h}\) (counting multiplicities) respectively. 
\begin{lemma}\label{Lemma 2}
    Given  integers $m\geq 0$ and $n\geq 1$, there exist $C,b_0>0$ such that, if \(b\geq b_0\), then,
    \[ \lambda_n\paren{H_{m,b}^N} \geq \min \{\lambda_n\paren{\hat{H}_{m,b}}, \lambda_n\paren{\tilde{H}_{m,b}}\} -C.\]
\end{lemma}
\begin{proof}
Consider a partition of unity \(\{\chi_1, \chi_2\}\) with \(\chi_k \in C^{\infty}(\R; [0,1])\), 
    \[\text{supp } \chi_1 \subset \biggl(-\infty, \frac{3}{4}\biggr),\quad \text{ supp } \chi_2 \subset \biggl(\frac{1}{2}, +\infty\biggr),\]
    and \(\chi_1^2 + \chi_2^2 = 1\). 
     Since \(\chi_1 u, \chi_2 u \in \mathcal{D}\paren{H^N_{m,b}}\), we can apply the IMS localization formula \cite[Theorem 3.2]{MR0883643}, for $u\in\mathcal{D}(H_{m,b}^N)$, 
    \[q_{m,b}(u) =\hat q_{m,b}(\chi_1u)+\tilde q_{m,b}(\chi_2u)- \int_{0}^1 (\abs{\chi'_1}^2 +\abs{\chi'_2}^2)\abs{u(r)}^2\,rdr. \]
    Since \(\chi_1 u\in \mathcal{D}\paren{\hat{q}_{m,b}}\) and \(\chi_2 u\in \mathcal{D}\paren{\tilde{q}_{m,b}}\), we get by the min-max principle
    \[ \lambda_n\paren{{H}^N_{m,b}} \geq \lambda_n\paren{\hat{H}_{m,b}\oplus\tilde{H}_{m,b}}-C,\]
    with  \(C=\abs{\chi'_1 }^2_\infty+\abs{\chi'_2}^2_\infty\).
\end{proof}

Now we introduce the self-adjoint operator in $L^2(\R_+,rdr)$,
\begin{equation}\label{eq:app-op-Lm}
\begin{aligned}
    \mathrm{L}_{m,b}&=-\frac{d^2}{dr^2}-\frac1r\frac{d}{dr} + \Bigl(\frac{m}{r}-\frac{br}2\Bigr)^2,\\
    \mathcal{D}\paren{\mathrm{L}_{m,b}}& = \{u\colon u,u/r,u',  \mathrm{L}_{m,b}  u\in L^2(\R_+, rdr),~\lim_{r\to0^+}\frac{u(r)}{\ln r}=0\mbox{ for }m=0\},
\end{aligned}
\end{equation}
 and associated to the quadratic form
\begin{equation}\label{eq:app-def-op-lm}
\begin{aligned}
     \ell_{m,b}(u)& = \int_{0}^{+\infty} \biggl( \abs{u'(r)}^2  +\biggl(\frac{m}{r}-\frac{br}{2}\biggr)^2 \abs{u(r)}^2 \biggr) \, rdr,\\
     \mathcal{D}\paren{\ell_{m,b}}&= \{u\colon u,mu/r,u'\in L^2(\R_+, rdr)\}.
\end{aligned}
\end{equation}
\begin{lemma}\label{Lemma 3}
    Given integers $m\geq 0$ and $n\geq 1$, we have for $b>0$,
    \begin{equation}
        \lambda_n\paren{\hat{H}_{m,b}} \geq \lambda_n \paren{\mathrm{L}_{m,b}}.
    \end{equation}
\end{lemma}
\begin{proof}
    This follows from the min-max principle, since every \(u\) in the form domain of \( \hat{H}_{m,b}\) can be extended by zero to the positive half-axis, and  this extension is in the form domain of \(\mathrm{L}_{m,b}\). 
\end{proof}

\begin{proof}[Proof of Proposition~\ref{Proposition 1}]
By Sturm-Liouville theory, the eigenvalues of $\mathrm{L}_{m,b}$ are simple, hence
\[ 0< \lambda_1(\mathrm{L}_{m,b})<\lambda_2(\mathrm{L}_{m,b})<\cdots.\]
Moreover, by separation of variables, \(\{\lambda_n(\mathrm{L}_{m,b})\}_{n\in \mathbb{N}}\) are eigenvalues of the Landau Hamiltonian on $\R^2$,
\[ \mathrm{L}_b=(-i\nabla-bA)^2.\]
We know that the eigenvalues of $\mathrm{L}_b$ have infinite multiplicity and they are precisely given by the (strictly) increasing sequence,
\[\lambda_n(\mathrm{L}_b)=(2n-1)b\quad \text{ with } n= 1,2,3,\ldots.\]
Knowing that $\lambda_1(\mathrm{L}_{m,b})$ is an eigenvalue of $\mathrm{L}_{b}$, we have $\lambda_1(\mathrm{L}_{m,b})\geq \lambda_1(\mathrm{L}_{b})=b$. Then, knowing that $\lambda_2(\mathrm{L}_{m,b})$ is an eigenvalue of $\mathrm{L}_{b}$ and
$\lambda_2(\mathrm{L}_{m,b})>\lambda_1(\mathrm{L}_{b})$, we deduce that $\lambda_2(\mathrm{L}_{m,b})\geq \lambda_2(\mathrm{L}_{b})=3b$. Continuing in this manner, we deduce that
\[ \lambda_n(\mathrm{L}_{m,b})\geq \lambda_n(\mathrm{L}_{b})=(2n-1)b\quad (n\geq 1).\]
To conclude, we use Lemmas~\ref{Lemma robin boundary},\ref{Lemma 2} and \ref{Lemma 3} and the fact that \((2n-1)b <\frac{b^2}{64}\) for \(b\) large enough.
\end{proof}

\bibliographystyle{alpha} 
\bibliography{mybibliography} 
\end{document}